\documentclass[12pt, preprint, reqno]{amsart}
\usepackage{graphicx}
\usepackage[dvipsnames]{xcolor}

\usepackage{tikz}
\usetikzlibrary{calc,intersections}
\usepackage{graphicx}

\usepackage{amsfonts, amsmath, amssymb}
\usepackage{amsrefs}
\makeatletter
\renewcommand{\MR}[1]{}
\makeatother

\usepackage{amsthm}
\usepackage{hyperref}
\hypersetup{colorlinks={true},linkcolor={black},citecolor=black}

\usepackage{mathrsfs}

\usepackage{enumitem}

\usepackage[normalem]{ulem}

\newcommand{\characteristic}[1]{1_{#1}}

\usepackage{geometry}

\usepackage[scr=esstix]{mathalfa}

\usepackage{xspace}

\numberwithin{equation}{section}

\usepackage[capitalize]{cleveref}

\renewcommand{\epsilon}{\varepsilon}
\newcommand{\dist}{\mathrm{dist}}

\newcommand{\diam}{\mathrm{diam}}

\newcommand{\p}{\partial}

\newcommand{\R}{{\mathbf R}}

\newtheorem{theorem}{Theorem}[section]
\newtheorem{lemma}[theorem]{Lemma}
\newtheorem{corollary}[theorem]{Corollary}

\newtheorem{proposition}[theorem]{Proposition}

\theoremstyle{definition}
\newtheorem{remark}[theorem]{Remark}

\title[Hausdorff dimension
of the monopolist's free boundary curve]{On
the Hausdorff dimension and singularities of the monopolist's free boundary curve
}
\author{Robert J. McCann}
\address{Departments of Mathematics and Economics, University of Toronto}
\email{mccann@math.utoronto.ca}
\author{Lucas D. O'Brien}
\address{Department of Mathematics, Massachusetts Institute of Technology}
\email{obrie720@mit.edu}
\author{Cale Rankin}
\address{School of Science, University of New South Wales Canberra}
\email{c.rankin@unsw.edu.au}

\thanks{
\\ 
$^*$Robert McCann's work was supported in part by the Canada Research Chairs program CRC-2020-00289 and Natural Sciences and Engineering Research Council of Canada Discovery Grants RGPIN--2020--04162. 
Lucas O'Brien's work was supported in part by a University of Toronto Excellence Award and by an internal fellowship from the Massachusetts Institute of Technology School of Science.
Cale Rankin's work was supported by grants DP220100067 \& DE260100829 of the Australian Research Council.
The authors are grateful to Shibing Chen, Alessio Figalli, Yi Ru-Ya Zhang, and Kelvin Shuangjian Zhang for fruitful exchanges. \\
    \copyright \today}

\begin{document}

   \begin{abstract} 
The simplest genuinely multidimensional monopolist's problem involves  minimizing a linearly perturbed Dirichlet energy among nonnegative convex functions $u$ on an open domain $X \subset [0, \infty)^2$.  The
geometry of the region of strict convexity $\Omega\subset X$ for the unique minimizer $u$ is of central interest.  A relatively
closed portion $X_1^0 \subset X$ of the domain is comprised of line segments starting and ending on $\p X$ along which $u$ is affine.  
For convex polygons and potentially all domains $X \subset \R^2$,   we build on results with Zhang to show that outside $X_1^0 \cup \{u=0\}$,  the free boundary of $\Omega$ is a continuous curve of Hausdorff dimension one,  and that $\Omega$ has density $1/2$ along it (and is $C^\alpha_{\mathrm{loc}}$ for all $0<\alpha<1$), except perhaps at a discrete set of singular points.
We do this by showing that much of the free boundary solves an obstacle problem
whose endogenous obstacle is $C^2$.

From a slightly stronger conclusion,
we deduce the free boundary becomes locally $C^\infty$ outside a closed set whose relative interior is empty. 
In response to the circulation of the present manuscript, 
we received a concurrent but independent work of Chen, Figalli and Zhang
who verify a strengthening sufficient for this partial regularity result; (they show in particular that $\alpha=1$ and the discrete set mentioned above is empty).
\end{abstract}
\maketitle

\section{Introduction}

The problem of identifying optimal strategies facing asymmetric information has a venerable history in economic theory, surveyed in Laffont and Martimort \cite{LaffontMartimort02} and Basov \cite{Basov05}. 
The monopolist's problem  in the principal agent framework of Mussa and Rosen \cite{MussaRosen78} serves as a paradigm in this regard:
a monopolist wishes to price a line of products so as to optimize profits,  facing a field of anonymous buyers whose preferences are known only statistically.
The seminal works of Mirrlees \cite{Mirrlees71},
Spence~\cite{Spence73},
and Myerson \cite{Myerson81}
address variants of this problem in which the spaces --- both of actions and of private information --- are again unidimensional.  In spite of much progress --- see references in \cites{Wilson91, Armstrong96, DaskalakisDeckelbaumTzamos17, KolesnikovSandomirskiyTsyvinskiZimin22+, Kolesnikov23+} and below ---
the multidimensional version of the problem continues to pose vexing challenges.  

A landmark contribution of Rochet and Chon\'e 
\cite{RochetChone98} reduced the question to a problem in the calculus of variations under the simplifying assumptions of linear transfers and quadratic direct costs and preferences.  Under the further assumption 
that the space of actions is
$Y=[0,\infty)^n$,  and that types are uniformly distributed over a domain $X \subset Y$ of the same dimension, this takes the form of maximizing the profit functional
\begin{equation}\label{eq:objectivefunctionaldefinition}
\Pi(u):= \int_{X}[x \cdot \nabla u(x)- u(x) - \frac12 |\nabla u(x)|^2] \ d\mathcal{H}^n
\end{equation}
over the set 
\begin{equation}\label{eq:constraintset}
    \mathcal{U} := \{0 \le u \in W^{1,2}(\overline{X}) \ | \ u \text{ is convex} 
    \}
\end{equation}
of convex indirect utilities; here $\mathcal{H}^k$ denotes $k$-dimensional Hausdorff measure. If $X$ is convex,
the product selected turns out to depend continuously on consumer type $x \in X$ \cites{RochetChone98, CarlierLachand--Robert01}, and to be given by $\nabla u(x)$,  
where $u$ is the unique maximizer of 
\eqref{eq:objectivefunctionaldefinition}--\eqref{eq:constraintset}.
This continuity also extends to $\overline X$ \cite{ChenFigalliZhang26+}.
Generalizations to other distributions of types
\cite{RochetChone98}, preferences \cites{Rochet87, Carlier01, FigalliKimMcCann11}  and transfers \cites{NoldekeSamuelson18, McCannZhang19},
including regularity results \cites{Chen23, McCannRankinZhang25}, have also been given e.g.
by Carlier, Chen, Figalli, Kim, McCann, Noldeke, Rankin, Rochet, Samuelson, Zhang,
and to auction design by Daskalakis, Deckelbaum and Tzamos \cite{DaskalakisDeckelbaumTzamos17}.
Rochet--Chon\'e showed the maximizer of \eqref{eq:objectivefunctionaldefinition}--\eqref{eq:constraintset} was characterized by the convex--ordering of the positive and negative parts of the variational derivative,  conditioned on the action selected. A duality-based characterization \cites{DaskalakisDeckelbaumTzamos17, McCannZhang23+} was eventually parlayed into a characterization via
Euler-Lagrange equations for a free boundary problem by two of us with Zhang~\cite{McCannRankinZhang24+}.   
However, none of these characterizations extend to the more general settings of
\cites{Carlier01, FigalliKimMcCann11, NoldekeSamuelson18, McCannZhang19}.

For convex domains $X \subset \R^n$,  a deeper analysis of the maximizer of \eqref{eq:objectivefunctionaldefinition}--\eqref{eq:constraintset}
performed by two of us with Zhang \cite{McCannRankinZhang24+}
revealed unexpected structures, only partly
foreshadowed by numerics of Rochet--Chon\'e \cite{RochetChone98}, Ekeland--Moreno-Bromberg \cite{EkelandMoreno-Bromberg10}, Mirebeau \cite{Mirebeau16}, Boerma--Tsyvinsky--Zimin  \cite{BoermaTsyvinskiZimin22+}, Carlier--Dupuis--Rochet--Thanassoulis~\cite{CarlierDupiusRochetThanassoulis24} and theory of McCann and Zhang \cites{McCannZhang24+, McCannZhang23+}.  
For $n=2$ the analysis of the free boundary separating
the customization region --- the largest open set $\Omega_n\subset X$ on which $u$ is strictly
convex --- from its complement was initiated in~\cite{McCannRankinZhang24+};
it will be continued here, still for $n=2$.  Line segments  
along which $u$ is affine that have both endpoints on $\p X$ 
form a relatively closed subset of $X$.
For domains satisfying a family of regularity assumptions (see \S 1.1(a)--(d)) --- which include all convex polygons
$X \subset [0,\infty)^2$ ---
we shall show that outside of this subset, 
the free boundary $X \cap \p\Omega_2$ has Hausdorff dimension 1, 
and  $\Omega_2$ has density $1/2$
along $X \cap \p\Omega_2$,  apart from a discrete set of singular points.
This is achieved by building on our results 
with Zhang \cite{McCannRankinZhang24+}
and combining them with techniques 
pioneered for the obstacle problem
with $C^{2,\alpha}$ obstacles
by Caffarelli \cite{Caffarelli77,Caffarelli98}, and
extended to lower regularity obstacles by 
Blank \cite{Blank01}.
Since our obstacle arises endogenously from $u$
and was not known to be better than $C^{1,1}$ previously,  we use \cite{McCannRankinZhang24+}*{Lemma 7.3}  and
its proof to show the obstacle is $C^2$ 
and control the singular points on the free boundary before analyzing the regular points (where $\Omega_2$ has density $1/2$) by exploiting developments in Blank \cite{Blank01}.

A key to understanding the maximizer $u$ of \eqref{eq:objectivefunctionaldefinition} on \eqref{eq:constraintset} is to divide the domain $X$ into the convex sets where the  solution $u$ to the monopolist's problem contacts each of its supporting hyperplanes 
\cite{RochetChone98}. Define an equivalence relation on the closure $\overline X$ of $X$ by setting 
\begin{equation}\label{eq:equivalencerelationdefinition}
    x_1 \sim x_2 \text{ when } \nabla u(x_1) = \nabla u(x_2).
\end{equation} 
Given $x \in \overline X$, denote by $\tilde{x}$ its equivalence class under this relation. Following the terminology of \cite{McCannRankinZhang24+}, we will refer to the convex sets $\tilde{x}$ as \textit{leaves}, or as \textit{rays} if they are $1$-dimensional. We now group points in our domain $X$ based on the dimension of their associated leaves. In particular, we define 
\begin{equation}\label{Omega_i}
\Omega_i := \{x \in \overline{X} \ | \ \tilde{x} \text{ is }(n-i)\text{-dimensional}\} 
\end{equation}
and $X_i := X \cap \Omega_i$. Since $X \subset \R^2$ in the sequel, we can write
\[
\overline{X} = \Omega_0 \cup \Omega_1 \cup \Omega_2. 
\]
When nonempty, the {\em exclusion region} $\Omega_0$ coincides with the set of buyers who are priced out of the market,  the {\em customization region} $\Omega_2$ represents the set of buyers who choose customized products,  and the {\em bunching region} $\Omega_1$ is foliated by rays, with the buyers on each ray choosing the same product.
On each of the regions $\Omega_i$, there is a description of $u$ from the Euler-Lagrange equations for the monopolist's problem. Firstly, by \cite{McCannRankinZhang24+}*{Theorem 1.1}, whenever $\Omega_0 \ne \emptyset$, then $\Omega_0 = \{x \in \overline{X} \ | \ u(x) = 0\}$; in particular, $\Omega_0$ is closed and convex. On $\Omega_2$, $u$ is strictly convex, and by \cite{McCannRankinZhang24+}*{Theorem 1.1}, $\Delta u = 3$ on $\Omega_2$, as anticipated by Rochet and Chon\'e \cite{RochetChone98}.  On $\Omega_1$, the description of $u$ is more complicated: for now, note that by \cite{McCannRankinZhang24+}*{Theorem~1.1}, 
for each $x \in \Omega_1$, the equivalence class $\tilde{x}$ intersects the fixed boundary $\partial X$ in either one, two, or {infinitely} many points, and $u$ is affine along $\tilde{x}$.
Proposition 2.3 of \cite{McCannRankinZhang24+} shows the outward normal distortion
is nonnegative: $(\nabla u(x)-x)\cdot \mathbf n \ge 0$
throughout $\p X$.
Crucially, the boundaries between the regions $\Omega_0, \Omega_1,$ and $\Omega_2$ are determined as part of the solution; therefore, the monopolist's problem can be studied as a free-boundary problem. 

The most interesting free boundary to study in this problem is $\partial \Omega_1 \cap \partial \Omega_2$.  Indeed, since $\Omega_0$ is convex, its boundary is Lipschitz and the regularity of the free boundaries $\partial \Omega_0 \cap \partial \Omega_1$ and $\partial \Omega_0 \cap \partial \Omega_2$ is comparatively easy to understand. In fact, by our work with Zhang \cite{McCannRankinZhang24+}*{Theorem 1.1}, \cite{McCannRankinZhang25}*{Theorem 3}, any point in $(X_0 \cap X_2) \setminus \overline{X_1}$ has a neighbourhood on which $\frac{1}{3}u$ is a convex $C^{1,1}$ solution to the \textit{classical obstacle problem} $\Delta(\frac{1}{3}u) = \characteristic{\{u > 0\}}$, and therefore by the classical theory of the obstacle problem \cite{Caffarelli98}, $(X_2 \cap X_0) \setminus \overline{X_1}$ is locally analytic
\cite{McCannRankinZhang24+}*{Remark 1.2}. Similarly, the free boundary $\partial \Omega_1 \cap \partial \Omega_2$ will be studied via an obstacle problem, in which the obstacle is the minimal convex extension of $u|_{\Omega_1}$. However, since an unpublished result of Caffarelli and Lions   shows $u$ is only $C^{1,1}_{\mathrm{loc}}$ in general 
(see \cite{McCannRankinZhang25}*{especially Remark 5} --- this result
was improved to $C^{1,1}$ on certain domains in  \cite{McCannRankinZhang24+}*{Theorem 4.1}), Caffarelli's regularity theory for the obstacle problem~\cite{Caffarelli98}, which requires the obstacle to be at least $C^{2, \alpha}$, no longer applies, and we must appeal to Blank and Hao's ideas instead~\cites{Blank01,BlankHao15}.

\subsection{Statement of main results}

Let $V \subset \p X$ be the set of points where $\p X$ is not smooth, 
$\mathbf{n}$ denote the outer normal to $\p X$ outside $V$, and use \cite{McCannRankinZhang24+}*{Theorem 1.1 and Proposition 2.3} to 
partition the bunching region as follows:
\begin{align}\label{Omega1+}
\Omega_1^+ &:= \{x \in \Omega_1 \mid \tilde x \cap \p X=\{x_0\}\ \mbox{\rm with}\ (\nabla u(x_0)-x_0)\cdot \mathbf n >0\ \mbox{and}\ x_0 \not\in V \}, 
\\ \nonumber \Omega_1^- &:= \{x \in \Omega_1 \mid \tilde x \cap \p X=\{x_0\}
\ \mbox{\rm with}\  
(\nabla u(x_0)-x_0)\cdot \mathbf n =0 
\ \mbox{\rm or}\ x_0 \in V\},
\\ \nonumber  \Omega_1^0 &:= \{x \in \Omega_1 \mid \#(\tilde x \cap \p X) \ge 2\}.
\end{align}
Set $X_1^{\pm } := X \cap \Omega_1^{\pm}$ and
observe that $X_1^0 := X \cap \Omega_1^0$ is a relatively closed subset of $X$.
Letting $|A|=\mathcal{H}^2(A)$ denote the Lebesgue measure of $A \subset \R^2$,
we make the following assumptions on our domain 
and the corresponding solution $u$ to the monopolist's 
optimization of \eqref{eq:objectivefunctionaldefinition} 
on \eqref{eq:constraintset}:
\begin{itemize}\label{ie}
\item[(a)]{
$X \subset  \R^2$ is open, bounded, and convex;} 
\item[(b)]{
$\p X$ is smooth except at finitely many points 
$\{x_1,\ldots,x_N\} =: V$;}
\item[(c)]\label{boundary regularity hypothesis}
{$u \in C^1(\overline X)$ and $u \in C^{1,1}_{\mathrm{loc}}(\Omega_1^+)$. }
\end{itemize}
For Corollary \ref{C:mainresults} only, 
we invoke the additional hypothesis:
\begin{itemize}
\item[(d)]{$|X_1^-| = 0$ and $\overline {X_1^-}\cap \p \Omega_2$  has a countable intersection with $X\setminus (X_0 \cup X_1^0)$;}
\item[(e)] {$\dim_{\mathcal H} (X_1^0 \cap \p \Omega_2) \le 1$.}
\end{itemize}
Assumptions (a)--(c) are satisfied on all
convex domains \cite[Theorem 1.1]{ChenFigalliZhang26+}  \cite[Theorem 4.1]{McCannRankinZhang24+}, while (d) holds at least on convex polygons 
$X \subset [0,\infty)^2$
\cite[Remark 8.4]{McCannRankinZhang24+}.
Since a reflection and strict maximum principle argument \cite[Lemma 3.2]{McCannRankinZhang24+} for Poisson's equation implies $\#(\tilde x \cap \p \Omega_2) \le 1$ for each $x \in X_1^0$ in all the examples that we are aware of, 
it is natural to conjecture (e) holds quite generally.  We have not verified this conjecture unless $X=(a,a+1)^2$, in which case $\#(X_1^0 \cap \p \Omega_2) \le 1$ by Remark~\ref{R:square}.

Our analysis focuses primarily 
on the {\em tame} free boundary points
$\mathcal{T}:=X_1^+ \cap \p \Omega_2$,  and its partition (following Caffarelli \cite{Caffarelli98} and Blank \cite{Blank01}) into
{\em regular points} $\mathcal R$ and {\em singular points} $\mathcal S$ defined by:
\begin{align}
\label{eq:DefRegular}
\mathcal{R} &:= \{x \in \mathcal{T}:= X_1^+ \cap \p\Omega_2 \mid \lim_{r \to 0} \frac{|\Omega_2 \cap B_r(x)|}{|B_r(x)|} = \frac12\},
\\ \mathcal{S} &:= \mathcal{T} \setminus \mathcal{R}.
\label{eq:DefSingular}
\end{align}

\begin{theorem}[Main results]\label{thm:mainresults}
    Let $u$ maximize the monopolist's profits \eqref{eq:objectivefunctionaldefinition}--\eqref{eq:constraintset} on an open bounded convex $X \subset \mathbf{R}^2$ satisfying (a)--(c) and define   the regular $\mathcal R$ and singular $\mathcal S$ parts of the tame free boundary $\mathcal{T}:= X_1^+ \cap \p \Omega_2$ by \eqref{eq:equivalencerelationdefinition}--\eqref{eq:DefSingular}.
    Then $u \in C^2(X_1^+)$, $\mathcal T$ is relatively open in $\p \Omega_2$, and
    \begin{enumerate}
        \item\label{item:singularsetisdiscrete} (Singular set is discrete) the singular set $\mathcal{S} = \mathcal{T}\setminus \mathcal{R}$ is a discrete subset of $\mathcal{T}$, and in particular is countable;
        \item\label{item:sharpdimensionbound}(Sharp Hausdorff dimension bound) \[
        \mathrm{dim}_{\mathcal{H}}(\mathcal{T}) \leq 1;
        \]
        \item\label{item:holderregularity} (Regular points form bi-H\"older curves) for all $\alpha \in (0,1)$,         each connected component of $\mathcal{R}$ is 
        locally $\alpha$-biH\"{o}lder equivalent to 
        the open unit interval;
        \item\label{item:Dini partial regularity}(Dini condition implies partial regularity) if $\Delta u$
        is Dini continuous on $X^+_1$, 
        then $\mathcal T$ is smooth
        outside of a closed subset whose relative interior is empty --- and moreover, each connected component of $\mathcal R$ is locally $C^1$-smooth.
        \end{enumerate}
\end{theorem}

This theorem improves the results of \cite{McCannRankinZhang24+}. There 
$\mathrm{dim}_{\mathcal{H}}(\mathcal{T}) <2$ was proved and the singular set $\mathcal{S}$ was identified but no control on its size was asserted.  The discreteness established here echoes a result proved by Monneau for smoother obstacles~\cite{Monneau03}.
Proposition~\ref{lemma:C^2obstacle} proves $u \in C^2(\Omega_1^+)$,
while Theorem \ref{thm:mainresults}\eqref{item:singularsetisdiscrete} is proved in Theorem~\ref{theorem:singularsetisdiscrete},   \eqref{item:sharpdimensionbound} and \eqref{item:holderregularity} are proved in Corollary~\ref{cor:holderregularityandsharpdimensionbound},
and \eqref{item:Dini partial regularity} is proved in
Proposition \ref{prop:conditionalregularity}.
The set $\mathcal T$ forms a relatively open subset of $\p \Omega_2$ due to assumptions (b)--(c).
Under the additional hypothesis (d) 
we can deduce the following corollary by combining 
Theorem \ref{thm:mainresults} with results of \cite{McCannRankinZhang24+}.
Its proof can be omitted on first reading.

\begin{corollary}
    [Free boundary of customization region]\label{C:mainresults}
   Let $u$ maximize the monopolist's profits \eqref{eq:objectivefunctionaldefinition}--\eqref{eq:constraintset} on a domain $X \subset [0,\infty)^2$ 
   with $V$    satisfying (a)--(d) 
   where 
   $X_i^{(+,-,0)} := X \cap \Omega_i^{(+,-,0)}$ as in 
   \eqref{Omega_i}--\eqref{Omega1+}.
    Then,
    \begin{enumerate}     

    {   \item\label{item+:countablylipschitz}
(Sharp Hausdorff dimension bound)        
        $\dim_{\mathcal H}\big((\p \Omega_2)\setminus        X_1^0\big) \le 1$; }

  { \item\label{item+:singularsetisdiscrete}
(Singular set is countable) 
the closure $\overline{\mathcal{S}'}$ of $\mathcal{S}' :=  \p X_2 \setminus \mathcal{R}'$
has at most a countable intersection with $X\setminus X_1^0$,
where
\begin{align*}
\mathcal{R}' :=\{x \in \p \Omega_2 \mid \lim_{r \to 0} \frac{|\Omega_2 \cap B_r(x)|}{|B_r(x)|} = \frac12\};        
\end{align*}      }     
{  \item\label{item+:quantify fixed}(Quantifying accumulation on fixed boundary) $\overline{\mathcal S'} \cap \p X \subset V \cup \overline{X \cap \p \Omega_2}.$ }      
{\item[(4)]\label{item+:sharpdimensionbound+}
(Extended bounds)
If, in addition, (e) holds then
$\mathrm{dim}_{\mathcal{H}}(\p \Omega_2) \leq 1$;

}
{\item[(5)]\label{item+:quantify fixed++}
if $X_1^0 
\cap \p \Omega_2$
is at most countable, the same is true for $X \cap \overline{\mathcal{S}'}$.}
\end{enumerate}
\end{corollary}

\begin{proof}
\textit{Proof of (1).} To begin, note the decomposition of $\overline X$ as the union $\overline X = {X_0 \cup X_1^- \cup X_1^0 \cup X_1^+\cup X_2 \cup \p X}$ yields a partition of $\partial \Omega_2 \setminus X^0_1$ into components $X_0 \cap \p \Omega_2$, $X_1^\pm \cap \p \Omega_2$,  and $\p X \cap \p \Omega_2$ since  \cite[Theorem 1.1]{McCannRankinZhang24+} shows $X_2$ is open. We will prove each of these components has Hausdorff dimension $1$. First, $X_1^+ \cap \partial \Omega_2$ has Hausdorff dimension equal to $1$ by Theorem \ref{thm:mainresults}(\ref{item:sharpdimensionbound}). Next, the boundaries of the convex sets $X$ and $\Omega_0=\{u=0\}$ are each covered by Lipschitz curves. Finally assumption (d) implies $(\overline{X_1^-} \cap \p \Omega_2) \setminus (X_0 \cup X^0_1)$ is countable.

\textit{Proof of (2).} To establish (2), set $FB := X \cap \p \Omega_2 \subset \p X_0 \cup \p X_1^+ \cup \p X_1^0 \cup \p X_1^-$;  we must show the set $\overline{\mathcal{S}'}$ of {\em singular points} in $FB \setminus  X^0_1$ 
is countable. We decompose $FB \setminus X^0_1$ into the following sets and show the intersection 
of each with $\overline{\mathcal{S}'}$ is countable:
\begin{align}
\label{eq:decomposition}  FB \setminus X^0_1 &\subset [FB \cap X_1^+]\cup [(FB \cap \overline{X_1^-}) \setminus (X_0 \cup X^0_1)]\cup [FB \cap
\overline{X_1^-} \cap \p \Omega_0] \\
\nonumber  &\quad \cup   [FB \cap \p \Omega_0 \setminus \overline{X_1^0 \cup X_1^-}] . 
\end{align}
First, the intersection of $\overline{\mathcal{S}'}$ with $X_{1}^+$ is the discrete set quantified by Theorem \ref{thm:mainresults}(\ref{item:singularsetisdiscrete}). 
Next, $(\overline{\mathcal{S}'} \cap \overline{X_1^-}) \setminus (X_0 \cup X^0_1)$ 
is closed and countable by hypothesis (d).    We claim the set $FB   \cap \overline{X_1^-}\cap \p \Omega_0$ is empty. For a contradiction take $x \in FB  \cap \overline{X_1^-}\cap \p \Omega_0$, then $x$ is a limit of stray rays with length bounded below and so lies in a line segment contained in the boundary of $\Omega_0$. Since at most one point in this line segment, namely the endpoint, lies in $\overline{X_1^+}$ and the stray rays have measure $0$, we see provided $x$ is not the endpoint of the line segment in $\partial\Omega_0$, $B_\epsilon(x) \setminus( \Omega_0 \cup \Omega_2)$ has measure $0$ for some $\epsilon > 0$. Thus,  a maximum principle and reflection argument, as in  \cite[Lemma 3.2]{McCannRankinZhang24+}, yields a contradiction.

The remaining subset of $\overline{\mathcal{S}'}$ to consider is elements in $U:= X\cap \p \Omega_0 \setminus \overline{X_1^0 \cup X_1^-}$. We first establish that $\Omega_0$, which is empty or satisfies $\Omega_0 = \{u=0\}$ by \cite[Theorem 1.1]{McCannRankinZhang24+}, has $\Omega_0 \cap \p X$ connected. This follows from considerations laid out in 
Rochet and Chon\'e \cite[Proposition 1]{RochetChone98}: reflection symmetry across the axes and concavity of the principals' maximization problem
implies $Du \in [0,\infty)^2$,  hence $X \cap \p \Omega_0$ is a graph over the antidiagonal of a concave curve with Lipschitz constant at most $1$. Next,  
$x \cdot \bf n \le 0$ along $\Omega_0 \cap \p X$ by \cite[Proposition 2.3]{McCannRankinZhang24+}, which implies the curve $\Omega_0 \cap \p X$
intersects only the portion of $\p X$ which is parameterized
as a convex curve over the antidiagonal.
Since the difference of a concave and convex function
which does not vanish on an interval can have at most two zeros,
it follows that $\Omega_0 \cap \p X$ is connected. 

Thus, we may now consider the potentially countably many components $U_i$ of 
$U:= X\cap \p \Omega_0 \setminus \overline{
X_1^0 \cup X_1^-}$, 
and establish that $FB \cap U$ has in fact at most two components. 
Indeed,  if $x_0 \in FB \cap U_i$ lies outside the (relative) interior of $FB \cap U_i$,  each neighbourhood of $x_0$ intersects $X_1$, and in fact $X_1^+$ 
since $x_0  \not\in \overline{X_1^0 \cup X_1^-}$.
Thus a sequence of rays starting at points $y_i \in \Omega_1^+ \cap \p X$ comes arbitrarily near to $x_0$. Since $x_0$ is interior to $X$ these rays have length bounded below and their limit is a line segment containing $x_0 \in \Omega_0$. Since $x_0\in \Omega_0$ implies $Du(x_0) = 0$, this line segment lies in the boundary of $\widetilde{x_0} = \Omega_0$, and extends to one of only two (by connectedness from preceding paragraph) endpoints $y_0$ of 
$\Omega_0 \cap \p X$.  {We shall confirm} 
that at most two such points $x_0$ can exist by proving the relative interior $(y_0,x_0)$ of this line segment is disjoint from $FB$. 
The analysis divides into two cases: either (i)  $y_0$ is the endpoint of an interval in $\Omega_1^+ \cap \p X$,  or (ii) a sequence of such intervals converge to $y_0$.
In case (i),  since  \cite[Lemma 7.4]{McCannRankinZhang24+} asserts continuity of $\diam(\tilde x)$ on such intervals,  it follows that one side of $(y_0,x_0)$ consists entirely of
$\Omega_1^+$ while the other side consists entirely of the convex set $\Omega_0$, so $(y_0,x_0)$ is disjoint from $\p \Omega_2$ as desired.  In case (ii),  if some 
$z \in (y_0,x_0) \cap \p \Omega_2$ it means $\diam(\tilde x)$ attains every value in $(|z-y_0|,|x_0-y_0|)$ in each neighbourhood of $y_0$ along $\p X$. 
But such oscillations can be ruled out as in the proof of \cite[Lemma 7.4]{McCannRankinZhang24+}.  Thus $U \cap FB$ consists of at most two connected components.

The relative interiors of these connected components are analytic curves disjoint from the singular set, by \cite[Remark 1.2]{McCannRankinZhang24+}. We've established each set on the right-hand side of \eqref{eq:decomposition} has a countable intersection with $\overline{\mathcal S'}$, thereby completing the proof of 
\eqref{item+:singularsetisdiscrete}. 

For the proof of (3) we note the 
endpoints of the (at most two) connected components of $FB \cap U$
are the only four points
in $\Omega_0$ that can lie in $\overline{\mathcal{S}'}$.  

Assertions (4)--(5) are straightforward consequences of (1)--(3) and the additional hypothesis (e). 
\end{proof}

\begin{remark}[Rochet and Chon\'e's square example \cite{RochetChone98}]\label{R:square}
In the case of the square $X =(a,a+1)^2$ with $a>0$, 
the preceding corollary yields $\mathrm{dim}_{\mathcal{H}}(\partial \Omega_2) \leq 1$
and $\overline{\mathcal{S}'}$
at most countable since 
$\#(X_1^0 \cap \p \Omega_2) \le 1$ 
and $\#(\p X \cap \overline{X \cap \p \Omega_2}) \le 2$
by \cite[Theorem 1.5 and Lemma 8.7]{McCannRankinZhang24+}.
The square's geometry models buyers whose two characteristics $x=(x^1,x^2)$ are distributed independently and uniformly over two intervals; the economic relevance of this has made it a prominent but challenging 
test case for
 numerics and theory 
\cite{EkelandMoreno-Bromberg10,Mirebeau16,BoermaTsyvinskiZimin22+,CarlierDupiusRochetThanassoulis24, McCannRankinZhang24+, McCannZhang23+, McCannZhang24+}.
\end{remark}

The following remark improves the conditional result of Theorem \ref{thm:mainresults}\eqref{item:Dini partial regularity} by showing $C^\infty$-regularity outside a topologically small set unconditionally.
%without requiring Dini continuity of $\Delta u$.

\begin{remark}[Note added in revision]\label{R:revision}
In response to the circulation of the present manuscript,
we received a concurrent but independent work of Chen, Figalli and Zhang \cite{ChenFigalliZhang26+} 
that, in addition to verifying the boundary regularity required for (c), 
improves on Theorem \ref{thm:mainresults}\eqref{item:singularsetisdiscrete},\eqref{item:holderregularity} for $\p X \in C^{1,1}$ showing that $\mathcal S =\emptyset$ and $\alpha=1$.  Under (a)--(b), this improved regularity combines with our proof of Proposition~\ref{prop:conditionalregularity} to yield a better conclusion than 
Theorem \ref{thm:mainresults}\eqref{item:Dini partial regularity}: namely, unconditional
smoothness of the tame free boundary $\mathcal T$ 
outside of a closed subset whose interior relative to $\mathcal T$ is empty
(independent of whether or not $\Delta u$ is Dini continuous throughout $X_1^+$).
\end{remark}

\section{The monopolist's free boundary problem in the plane}\label{sec:preliminaryresultsonmonopolist'sproblem}

\subsection{The obstacle problem at tame rays}\label{subsec:freeboundaryproblemattamerays}

In what follows, $u$ denotes the solution to the monopolist's problem, i.e. the maximizer of $\Pi(\cdot)$ on $\mathcal{U}$ from \eqref{eq:objectivefunctionaldefinition}--\eqref{eq:constraintset}. For each point $x$ in the free boundary \[\Gamma : = X\cap \partial \Omega_1 
\cap \partial \Omega_2,\] the equivalence class $\tilde{x}$ is a ray which intersects the fixed boundary $\partial X$ by \cite{McCannRankinZhang24+}*{Theorem 1.1}. Henceforth, we shall focus our attention on the tame free boundary points 
$\mathcal{T}=X_1^+ \cap \p \Omega_2$ of 
\eqref{Omega1+}--\eqref{eq:DefRegular},
and the rays $\tilde x$ for which $x \in \mathcal T$.
Such rays are called tame and defined by the properties that 
$ \tilde{x} \cap \partial X =\{x_0\}$ with $\p X$ smooth at $x_0$, and 
\begin{equation}\label{eq:neumanncondition}
    (\nabla u(x_0) - x_0) \cdot \mathbf{n} > 0.
\end{equation}

We now reduce the problem of studying the free boundary $\Gamma$ near a tame free boundary point to the obstacle problem, as was first done in \cite{McCannRankinZhang24+}*{Section 7.1}. We begin by recalling \cite{McCannRankinZhang24+}*{Lemma 6.1}, which allows us to find a neighbourhood of a tame ray which is foliated by tame rays. 

\begin{lemma}[Local foliation about each tame ray]\label{lemma:localfoliationabouttamerays}
     Let $x \in \Gamma$ be a tame free boundary point with $\{x_0\} = \tilde{x} \cap \partial X$. Then, there exist $\epsilon, r_0, \eta_0 > 0$ and a smooth unit-speed curve $\gamma:[-\epsilon, \epsilon] \to \partial X$ such that for each $t \in [-\epsilon, \epsilon]$,
     \begin{itemize}
         \item $\diam(\widetilde{\gamma(t)})> r_0$,
         \item $\widetilde{\gamma(t)} \cap \partial X = \{\gamma(t)\}$, and
         \item $(\nabla u(\gamma(t)) - \gamma(t))\cdot \mathbf{n} \geq \eta_0$,
     \end{itemize}
      where $\widetilde{\gamma(t)}$ denotes the equivalence class of $\gamma(t)$ under \eqref{eq:equivalencerelationdefinition}.
\end{lemma}

\begin{proof}
    By \cite{McCannRankinZhang24+}*{Lemma 6.1}, there is some $\epsilon > 0$, $r_0>0$, and smooth unit-speed curve $\gamma: [-\epsilon, \epsilon] \to \partial X$ such that $\diam(\widetilde{\gamma(t)}) \geq  r_0$ and $\widetilde{\gamma(t)} \cap \partial X = \{\gamma(t)\}$ for each $t \in [-\epsilon, \epsilon]$. 
    Moreover, since $u \in C^1(\Omega_1^+)$ by hypothesis (c) of Theorem \ref{thm:mainresults},
    $\partial X$ is smooth in a neighbourhood of $x_0$, and $(\nabla u(x_0) - x_0)\cdot \mathbf{n} > 0$ by the definition of tame \eqref{eq:neumanncondition}, if we take $\epsilon$ small enough then we may find some constant $\eta_0 > 0$ such that
   \[
   (\nabla u(\gamma(t)) - \gamma(t))\cdot \mathbf{n} \geq \eta_0
   \]
    for all $t \in [-\epsilon, \epsilon]$.
\end{proof}
From \cite{McCannRankinZhang24+}*{Section 6}, we see that there is a natural choice of coordinates for the local foliation in Lemma \ref{lemma:localfoliationabouttamerays}.
\begin{remark}[Construction of $(r,t)$ coordinates, \cite{McCannRankinZhang24+}]\label{remark:constructionof(r,t)coordinates}
     Take $\gamma(t)$ as in Lemma \ref{lemma:localfoliationabouttamerays}. For each $t \in [-\epsilon, \epsilon]$, define $\xi(t):= (\xi^1(t), \xi^2(t))$ to be the unit vector parallel to the leaf $\widetilde{\gamma(t)}$, pointing into $X$, and define $R(t) := \diam(\widetilde{\gamma(t)})$. By \cite[Theorem 7.3]{McCannRankinZhang24+}, $R(t)$ is continuous. Define 
     \[\mathcal{N} :=\bigcup_{t \in [-\epsilon, \epsilon]} \widetilde{\gamma(t)} \subset \Omega_1\]
     Then, for each $t$,
    \[
    \widetilde{\gamma(t)} = \{\gamma(t) + r\xi(t) \ | \ 0 \leq r \leq R(t)\},
    \]
    so we may parameterize $\mathcal{N}$ by 
    \begin{equation}\label{eq:ndefinition}
        \mathcal{N}= \{x(r,t) := \gamma(t) + r \xi(t) \ | \ t \in [-\epsilon,\epsilon], r \in [0, R(t)]\}.
    \end{equation}
    Following \cite{McCannRankinZhang24+}*{Section 7.1},  we use the preceding formula to
    extend $u|_{\Omega_1}$ to the domain 
    \begin{equation}\label{eq:nextdefinition}
    \mathcal{N}_{\mathrm{ext}} := \{x(r, t) = \gamma(t) + r\xi(t) \ | \ t \in [-\epsilon,\epsilon] , 0 \leq r < + \infty\}.
    \end{equation}
    By \cite{McCannRankinZhang24+}*{Remark 6.4}, the rays $\widetilde{\gamma(t)}$ spread out as they move away from the boundary, meaning that $(r, t) \mapsto x(r,t)$ is injective on $[-\epsilon, \epsilon]\times[0, \infty)$, thus these coordinates remain well-defined on the extension $\mathcal{N}_{\mathrm{ext}}$.
    
\end{remark}

Now, we are ready to study the free boundary $\Gamma$ using the obstacle problem. 
Throughout, given a domain $U \subset \mathbf{R}^n$ and a nonnegative function $v \in C^{1,1}(U)$ we define its \textit{contact set} 
 \begin{equation}\label{contact set}
     \Lambda(v) := \{x \in U \ | \ v(x) = 0\},
  \end{equation} 
its \textit{noncontact set }
\begin{equation}\label{noncontact set}
\Omega(v) := \{x \in U \ | \ v(x) > 0\},
\end{equation}
and its \textit{free boundary}
\begin{equation}\label{eq:free boundary}
F(v) := \Lambda(v) \cap \overline{\Omega(v)}. 
\end{equation}

\begin{lemma}[Reduction to the obstacle problem]\label{lemma:reductiontotheobstacleproblem}
      Suppose that $X \subset \mathbf{R}^2$ is a convex set satisfying hypotheses (a) and (c) of Theorem \ref{thm:mainresults}.    Let $x \in \Gamma$ be a tame free boundary point with $\{x_0\}= \tilde{x} \cap \partial X$. Take $r_0$ as in Lemma \ref{lemma:localfoliationabouttamerays} and the coordinates from Remark \ref{remark:constructionof(r,t)coordinates}, and define \begin{equation}\label{eq:rmaxdefn}
        R_{\mathrm{max}} := \max_{t \in [-\epsilon, \epsilon]}R(t).
    \end{equation}
    Let $\overline{R}_{\mathrm{max}} > R_{\mathrm{max}}$, and define \begin{equation}\label{eq:Udefn}
        U_{\epsilon} = U := \{x(r,t) \ | \ t \in (-\epsilon, \epsilon), r_0 < r < \overline{R}_{\mathrm{max}}\},
      \end{equation}
where, if necessary, a smaller choice of $\overline{R}_{\mathrm{max}} > R_{\mathrm{max}}$ ensures $\overline{U} \setminus \Omega_1 \subset \Omega_2$ and $U \Subset X$.
    Let $u_1$ be the (explicitly constructed) minimal convex extension of $u|_{\mathcal{N}}$ to $\mathcal{N}_{\mathrm{ext}}$. Then, $u_1 \in C^{1,1}(\overline{U})$, and letting $v:= u - u_1$, we have
    \begin{itemize}
        \item $\Lambda(v) = \overline{U} \cap \Omega_1$,
        \item $F(v) = U \cap \Gamma$,
        \item $v \in C^{1,1}(\overline{U})$,
        \item  $v \geq 0$, and
        \item$\Delta v \geq c_0 > 0$ holds on $\Omega(v):=\{v>0\}$ for some $c_0$.
    \end{itemize}
\end{lemma}
\begin{proof}
    Take the coordinates $(r,t)$ for $\mathcal{N}_{\mathrm{ext}}$ from Remark \ref{remark:constructionof(r,t)coordinates}. 
    Since each ray $\widetilde{\gamma(t)}$ is by definition a 1-dimensional set on which $u$ is in contact with its supporting hyperplanes, there exist functions $b, m : [-\epsilon,\epsilon] \to \mathbf{R}$ such that 
    \[
    u(x(r,t))= b(t) + rm(t)
    \]
    for $(r,t) \in[0, R(t)]\times[-\epsilon,\epsilon]$.
    Following \cite{McCannRankinZhang24+}*{Section 7.1}, we see that on $\mathcal{N}_{\mathrm{ext}}$, the minimal convex extension of $u|_{\mathcal{N}}$ is given by the formula
\begin{equation}\label{eq:convexextensiondef}
        u_1(x(r,t)) = b(t) + rm(t).
    \end{equation}
    By \cite{McCannRankinZhang24+}*{Lemma 6.3}, 
for a.e. $(r,t)$
    we have 
    \begin{equation}\label{Laplacian u1}
        3 - \Delta u_1 = \frac{3r - 2R(t)}{r + \frac{\xi(t) \times \dot{\gamma}(t)}{|\dot{\xi}(t)|}},
     \end{equation}
         where $\xi \times \dot{\gamma} = \xi^1 \dot{\gamma}^2 - \xi^2 \dot{\gamma}^1$. Note that $\xi(t) \times \dot{\gamma}(t) > 0$ by \cite{McCannRankinZhang24+}*{(64)}. Moreover, by \cite{McCannRankinZhang24+}*{(66)},
     \begin{equation}\label{Ray length}
         R(t)^2 |\dot{\xi}(t)|  = 2|\dot{\gamma}(t)| (\nabla u - x)\cdot \mathbf{n}.
    \end{equation}
    In particular, since $R(t) \geq r_0$ and $(\nabla u - x)\cdot \mathbf{n} \geq \eta_0$ for $t \in [- \epsilon, \epsilon]$, we see that $|\dot{\xi}(t)|$ is bounded below for $t \in [- \epsilon, \epsilon]$. So, on $\mathcal{N}_{\mathrm{ext}} \setminus \Omega_1$, we have 
    \begin{equation}\label{eq:obstaclelaplacianbound}
        \Delta u_1 \leq 3 - c_0,
    \end{equation}
    for some $c_0 > 0$ depending on $r_0$ and $\eta_0$. 

   Now, define $v:= u - u_1$. By definition, $u = u_1$ on $\overline{U} \cap \Omega_1$. Moreover, by the strict convexity of $u$ on $\overline{U} \setminus \Omega_1$ (since $\overline{U} \setminus \Omega_1 \subset \Omega_2$) and the fact that $u_1$ is the minimal convex extension of $u|_{\overline{U} \cap \Omega_1}$, we see that $u > u_1$ on $U \cap \Omega_2$. Therefore, 
   \[
   F(v) = U \cap \Gamma,
   \]
   and $v \geq 0$ on $U$. 
   
   By \cite{McCannRankinZhang25}*{Theorem 3}, we know that $u \in C^{1,1}(\overline{U})$. Moreover, by the fact that $u_1 \in C^{1,1}(\overline{U} \cap \Omega_1)$ and the expression for $u_1$ \eqref{eq:convexextensiondef}, it is clear that $u_1 \in C^{1,1}(\overline{U})$, thus $v \in C^{1,1}(\overline{U})$. Now, notice that 
    \[
    \Delta v = \Delta u - \Delta u_1 = 3 - \Delta u_1
    \]
    on $\Omega(v)$. By \eqref{eq:obstaclelaplacianbound}, this implies that
    \[
    \Delta v \geq c_0 >0
    \]
    on $\Omega(v)$.     
\end{proof}

    \subsection{Regularity of the obstacle}

    We now show that the obstacle $u_1$ gains two degrees of regularity compared to the continuous 
    function $R$ defined in Remark \ref{remark:constructionof(r,t)coordinates}. In particular, by the continuity of $R$, 
    $u_1 \in C^{2}(\overline{U})$.
    
\begin{remark}[The Heine-Cantor uniform continuity theorem]\label{R:Heine-Cantor}
    Recall that any function which is continuous on a compact set is uniformly continuous, and thus possesses a uniform modulus of continuity, by the Heine-Cantor theorem \cite{Engelking89}*{Theorem 4.3.32}.
    \end{remark}
    
    \begin{proposition}[Obstacle gains two degrees of regularity compared to $R(t)$]\label{lemma:C^2obstacle}
          Suppose that $X \subset \mathbf{R}^2$ is 
          a convex set satisfying hypotheses (a) and (c) of Theorem \ref{thm:mainresults}. Let $x \in \Gamma$ be a tame free boundary point. 
          %Suppose $\partial X$ is smooth in a neighbourhood of $\{x_0\}= \tilde{x} \cap \partial X$. 
          Let $u_1$ be the 
        minimal convex extension \eqref{eq:convexextensiondef}
        of $u|_{\mathcal{N}}$ to $U$, where $\mathcal{N}$ and $U$ are defined in \eqref{eq:ndefinition} and \eqref{eq:Udefn} respectively. Suppose that $R$ has modulus of continuity $\sigma$. Then, $D^2u_1$ has modulus of continuity $\rho(t)= C\max\{\sigma(t), t\}$ for some constant $C$. In particular, $u_1 \in C^2(\overline{U})$.
    \end{proposition}

     \begin{proof}
         Recall the Laplacian of $u_1$ is given \eqref{Laplacian u1} by
        \[
        \Delta u_1(x(r,t)) = 3 - \frac{3r - 2R(t)}{r + \frac{\xi(t) \times \dot \gamma(t)}{|\dot{\xi}(t)|}}.
        \]
        By \cite{McCannRankinZhang24+}*{Lemma 6.5}, $\xi(t)$ is Lipschitz. Moreover, from \eqref{Ray length}, 
        we have that
        \[
        |\dot{\xi}(t)| = \frac{2|\dot{\gamma}(t)|(\nabla u(\gamma(t)) - \gamma(t))\cdot \mathbf{n}}{R(t)^2}. 
        \]
        Since $u \in C^{1,1}(\overline{U})$ by Lemma \ref{lemma:reductiontotheobstacleproblem}, the fact that $R(t) \geq r_0 >0$ implies that $|\dot{\xi}|$ has modulus of continuity bounded by $C \max\{\sigma(t), t\}$ for some constant $C$; moreover, since $(\nabla u(\gamma(t)) - \gamma(t))\cdot \mathbf{n} \geq \eta_0 >0$ on $(- \epsilon, \epsilon)$, we have $ |\dot{\xi}(t)|$ bounded above and below as well. So, $\Delta u_1$ has modulus of continuity bounded by $C \max\{\sigma(t), t\}$ in $(r,t)$ coordinates, and so since the coordinate change between $(r,t)$ coordinates and Cartesian coordinates is bi-Lipschitz by \cite{McCannRankinZhang24+}*{Corollary 6.6}, we conclude that $\Delta u_1$ has modulus of continuity bounded by $C \max\{\sigma(t), t\}$ in Cartesian coordinates, for some possibly different constant $C$. 

         By the expression \eqref{eq:convexextensiondef}, it is clear that $\partial_{rr}^2 u_1 = 0$, since $u_1$ is affine in the $r$ direction. So, the Hessian of $u_1$ in Cartesian coordinates has a null eigenvalue, whose corresponding eigenvector is the ray with direction $\xi$. Since the eigenvectors of $D^2 u_1$ are orthogonal, we see that $\frac{\dot{\xi}}{|\dot{\xi}|}$ is also an eigenvector of $D^2 u_1$, on account of it being orthogonal to $\xi$ (since $|\xi(t)| = 1$ for each $t$). Therefore, we can write the Hessian of $u_1$ as \begin{equation}\label{eq:hessianexpression}
             D^2 u_1 = P^{-1}\mathrm{diag}(\Delta u_1, 0)P,
         \end{equation}
         where $P$ is the orthogonal matrix whose rows are the eigenvectors $\xi$ and $\frac{\dot{\xi}}{|\dot{\xi}|}$ of $D^2 u_1$. 

         We claim that $\xi$ and $\frac{\dot{\xi}}{|\dot{\xi}|}$ are Lipschitz in Cartesian coordinates. Indeed, $\xi(t)$ is Lipschitz \cite{McCannRankinZhang24+}*{Lemma 6.5}, and so since the change of coordinates between $(r,t)$ coordinates and Cartesian coordinates is bi-Lipschitz \cite{McCannRankinZhang24+}, $\xi$ is Lipschitz in Cartesian coordinates as well. Moreover, we claim $\frac{\dot{\xi}(t)}{|\dot{\xi}(t)|}$ is Lipschitz. Since $\xi(t)$ is Lipschitz, let $\zeta$ denote the (Lipschitz) normal vector to $\xi$ satisfying $\xi \times \zeta >0$. Since $|\xi(t)| = 1$ for each $t$, we see that $\frac{\dot{\xi}(t)}{|\dot{\xi}(t)|} = \pm \zeta(t)$ at every $t$. But from \cite{McCannRankinZhang24+}*{Remark 6.4}, we know that $\xi \times \dot{\xi} > 0$, therefore $\frac{\dot{\xi}(t)}{|\dot{\xi}(t)|} =  \zeta(t)$ is Lipschitz in $(r,t)$ coordinates, and hence in Cartesian coordinates as well. 

         Therefore, since $\xi$, $\frac{\dot{\xi}}{|\dot{\xi}|}$, and $\Delta u_1$ have modulus of continuity bounded by $C\max\{\sigma(t), t\}$ in Cartesian coordinates, we conclude by \eqref{eq:hessianexpression} that $D^2 u_1$ has modulus of continuity bounded by $C\max\{\sigma(t), t\}$ for some constant $C$ on $\overline{U}$. In particular,
         $R(t)$ is continuous 
         by Remark \ref{remark:constructionof(r,t)coordinates}, and thus uniformly continuous on the compact set $[-\epsilon, \epsilon]$ by the Heine-Cantor Theorem of Remark \ref{R:Heine-Cantor}. So, $u_1 \in C^{2}(\overline{U})$. 
    \end{proof}

\section{The low-regularity obstacle problem}\label{sec:low-regularityobstacleproblem}

Having reduced the problem of studying the tame free boundary to an obstacle problem, we now wish to study the obstacle problem 
\begin{equation}\label{eq:obstacleproblem}
    \Delta v = f\characteristic{\{v > 0\}}
\end{equation}
in the case that $f$ is only continuous. Recall that given a domain $U \subset \mathbf{R}^n$ and a nonnegative function $v \in C^{1,1}(U)$, we define its \textit{contact set} 
{
$\Lambda(v) =v^{-1}\{0\}$, 
\textit{noncontact set} $\Omega(v)= v^{-1}((0,\infty))$
and 
\textit{free boundary}
$F(v) := \Lambda(v) \cap \overline{\Omega(v)}$
as in \eqref{contact set}--\eqref{noncontact set}.
}

In his pioneering work on the obstacle problem (see \cite{Caffarelli98}), Caffarelli showed that when $f \in C^{\alpha}$, the free boundary of a solution to \eqref{eq:obstacleproblem} can be divided into a set of \textit{regular points}, near which the free boundary is $C^{1,\alpha}$, and a set of \textit{singular points}. Moreover, when $f \in C^{1,\alpha}$, he showed that the singular points are in fact contained in a $C^{1}$ submanifold. Blank \cite{Blank01} later showed that if $f$ is Dini continuous, then the free boundary is $C^{1}$-smooth near regular points. Moreover, Blank's initial regularity result is sharp; Blank gives an example \cite{Blank01}*{Theorem 7.3} in which $f$ is continuous, but the free boundary spirals infinitely upon itself and is not even locally the graph of a function. Nonetheless, Blank still shows that there exists a dichotomy between regular and singular points on the free boundary, which will be the starting point for our investigation into the monopolist's free boundary:
\begin{theorem}[Caffarelli alternative for $C^{2}$ obstacles \cite{Blank01}*{Theorem 4.3}]\label{theorem:blankcaffarellialternative} Let $U \subset \mathbf{R}^n$ be a domain and $c_0>0$. 
    Let $v \in C^{1,1}(\overline{U})$ be a nonnegative function such that $\Delta v =  f\characteristic{\{v > 0\}}$ for some $f \in C(\overline{U})$ satisfying $f \geq c_0 > 0$ on $\overline{U}$. Then, for any point $x_0 \in F(v) \cap U$, either:
    \begin{enumerate}[label=(\alph*)]
        \item\label{thm:defsingularpoint} $x_0$ is a \textit{singular point} of the free boundary, and
        \[
        \lim_{r \to 0}\frac{|\Lambda(v) \cap B_{r}(x_0)|}{|B_r|} = 0,\ \mbox{\rm or}
        \] 
        \item\label{thm:defregularpoint} $x_0$ is a \textit{regular point} of the free boundary, meaning
        \[
        \lim_{r \to 0}\frac{|\Lambda(v) \cap B_{r}(x_0)|}{|B_r|} = \frac{1}{2}.
        \]
    \end{enumerate}
\end{theorem}

\subsection{Blow-ups at free boundary points}\label{sec:blowupsatfreeboundarypoints}

As is typical of the obstacle problem, a crucial tool for our study of the free boundary will be blow-ups of $v$. Let $U \subset \mathbf{R}^n$ be a domain, and let $v \in C^{1,1}(\overline{U})$ be a nonnegative function satisfying $\Delta v \geq c_0 > 0$. Take $x_0 \in F(v)$ to be a point in the free boundary. For convenience, choose coordinates so that $x_0 = 0$.
For any $r > 0$, set
\begin{equation}\label{eq:rescalingdefinition}
    v_{r}(x) := \frac{1}{r^2}v(rx). 
\end{equation}
Then, for any sequence $\{s_k\}_{k=1}^{\infty}$ of nonnegative real numbers with $s_k \to 0$, the fact that $v \in C^{1,1}(\overline{U})$ and the Arzela-Ascoli theorem implies that there is some globally defined $v_{\infty} \in C^{1,1}(\mathbf{R}^n)$ and a subsequence (which we will not relabel) such that 
\[
v_{s_{k}} \to v_{\infty}
\]
with
\[
\nabla v_{s_k} \to \nabla v_{\infty}
\]
uniformly on compact subsets of $\mathbf{R}^n$. Such a $v_{\infty}$ is called a \textit{blow-up of $v$ about $x_0$}.

The following lemma follows from \cite{Blank01}*{Theorem 4.1}, and is implicit in \cite{Blank01}. 
\begin{lemma}[Blank's measure stability for blow-ups \cite{Blank01}]\label{lemma:measurestabilityforblowups}
    Let $v_1, v_2 \in C^{1,1}(\overline{B}_1)$, where $B_1 \subset \mathbf{R}^n$ is the unit ball, and let $f_1 \in C(\overline{B}_1)$ and $f_2 \in \mathbf{R}$, with $f_1 , f_2 \geq \lambda > 0$. Suppose $\Delta v_{i} =  f_i\characteristic{\{v_i > 0\}}$ on $B_1$. Then, there is some $C$ depending only on $n$, $||f_1||_{L^{\infty}(B_1)}$, $f_2$, and $\lambda$ such that
    \[
    |\Lambda(v_1)\Delta \Lambda(v_2)| \leq C\big(||f_1 - f_2||_{L^{\infty}(B_1)} + \sqrt{||v_1 - v_2||_{L^{\infty}(B_1)}} \,\, \big),
    \]
    where $A \Delta B$ denotes the symmetric difference of sets, and $|A \Delta B|$ its volume.
\end{lemma}
\begin{proof}
  Take $w \in C^{1,1}(\overline{B}_1)$ to be the solution to
    \[
    \bigg \{ \begin{matrix}
        \Delta w = f_2\characteristic{\{w > 0\}} & \text{ on }B_1, \\
        w = v_1 &\text{ on }\partial B_1. 
    \end{matrix}
    \]
    By \cite{Blank01}*{Theorem 4.1}, we know that 
    \[
    |\Lambda(v_1) \Delta \Lambda(w)| \leq C||f_1 - f_2||_{L^{\infty}(B_1)},
    \]
    for $C$ depending only on $n$, $||f_1||_{L^{\infty}(B_1)}$, and $f_2$. 
    By \cite{Caffarelli81}*{Corollary 4} or \cite{Blank01}*{Theorem 3.6}, we have 
    \begin{equation}
      \label{eq:cr-query}
      |\Lambda(w) \Delta \Lambda(v_2)| \leq C\sqrt{||v_{2} - w||_{L^{\infty}(B_1)}}.
    \end{equation}
    But by \cite{Blank01}*{Theorem 2.7(a)}, we have
    \[
    ||w - v_2||_{L^{\infty}(B_1)} \leq C(n)||v_1 - v_2||_{L^{\infty}(B_1)}. 
    \]
    So, since
    \[
    \Lambda(v_1) \Delta \Lambda(v_2) \subset (\Lambda(v_1) \Delta \Lambda(w)) \cup (\Lambda(w) \Delta \Lambda(v_2)),
    \]
    we conclude that there is some $C$ depending only on $n$, $||f_1||_{L^{\infty}(B_1)}$, $f_2$ and $\lambda$ such that
    \[
    |\Lambda(v_1) \Delta \Lambda(v_2)| \leq C \big(||f_1 - f_2||_{L^{\infty}(B_1)} + \sqrt{||v_1 - v_2||_{L^{\infty}(B_1)}}\,\,\big).
    \]
\end{proof}

Now, we will recall the classification of blow-ups for the obstacle problem. This result is standard, and follows easily in the low-regularity case from the results of \cites{Blank01, Caffarelli98}; we will include the proof for completeness.

\begin{proposition}[Classification of blow-ups \cites{Blank01, Caffarelli98}]\label{prop:classificationofblow-ups}
    Let $U \subset \mathbf{R}^n$, and let $v \in C^{1,1}(\overline{U})$ be a nonnegative function such that $\Delta v =  \characteristic{\{v > 0\}} f$ for some $f \in C(\overline{U})$ satisfying $f \geq c_0 > 0$ on $\overline{U}$. Let $x_0 \in F(v)$, and let $v_{\infty}$ be a blow-up of $v$ about $x_0$.
    \begin{enumerate}[label=(\alph*)]
        \item If $x_0$ is a singular point, then there is some symmetric, positive semi-definite matrix $Q_{\infty} \in \mathbf{R}^{n\times n}$ with $\mathrm{tr}(Q_{\infty}) = 1$ such that
        \[
        v_{\infty}(x - x_0) = \frac{f(x_0)}{2}x^T Q_{\infty} x.
        \]
        \item If $x_0$ is a regular point, then there is some $\tau \in \mathbf{S}^1$ such that 
        \[
        v_{\infty}(x- x_0) = \frac{f(x_0)}{2}(x \cdot \tau)_+^{\ 2},
        \]
        where $(t)_+ := \max\{t, 0\}$. 
    \end{enumerate}
\end{proposition}
\begin{proof}
    Proposition \ref{prop:classificationofblow-ups}(b) is simply a restatement of \cite{Blank01}*{Theorem 2.10(c)}. 

    Now, suppose $x_0 \in F(v)$ is a singular point, and choose coordinates so that $x_ 0 = 0$. By \cite{Blank01}*{Theorem 2.10(b)}, $v_{\infty}$ is a global solution to the obstacle problem $v_{\infty} =  f(0)\characteristic{\{v_{\infty} > 0\}}$. In particular, $v_{\infty}$ is convex by \cite{Caffarelli98}*{Corollary 7}. Moreover, the definition of a singular point in Theorem~\ref{theorem:blankcaffarellialternative} combined with Lemma~\ref{lemma:measurestabilityforblowups} implies that $\mathcal{H}^{n}(\Lambda(v_{\infty})) = 0$. Since $\Lambda(v_{\infty})$ is convex, this implies that $\Lambda(v_{\infty})$ is contained in a hyperplane $\Gamma := \{x \in \mathbf{R}^n \ | \ x \cdot \tau = 0\}$ for some $\tau \in \mathbf{S}^{n-1}$. Since $v_{\infty}$ is $C^{1,1}$ and satisfies $\Delta v_{\infty} = f(0)$ except on a line, $v_{\infty}$ is a weak solution to $\Delta v_{\infty} = f(0)$, and therefore, by the Weyl lemma, is a classical solution to $\Delta v_{\infty}(x) = f(0)$  for each for each $x \in \mathbf{R}^{n}$. Moreover, since
$v_{\infty} \in C^{1,1}(\mathbf{R}^n)$, $\partial_{ij}v_{\infty}$ is a bounded harmonic function on $\mathbf{R}^n$, and thus is constant by Liouville's theorem. So, $v_{\infty}$ is a degree-two polynomial. We know that both $v_{\infty}$ and its gradient vanish at the origin, and since $v_{\infty}$ is convex, its Hessian is positive semi-definite. So, by taking 
    \[
    Q_{\infty}= \frac{1}{f(0)}D^2 v_{\infty}(0)
    \]
    we see $v_{\infty}$ must have the desired form. 
\end{proof}

Thanks to the ray structure of the monopolist's problem, we can in fact show that there is a unique blow-up at a singular point of the free boundary. 

\begin{corollary}[Unique blow-up at singular points for the monopolist's problem]\label{cor:uniqueblow-upsatsingularpoints}
    Let $v$, $U$, and $u_1$ be defined as in Lemma \ref{lemma:reductiontotheobstacleproblem}, and let $x_0 \in F(v)$ be a singular point. Then, $v$ has a unique blow-up at $x_0$, and 
    \[
    v_{\infty}(x- x_0) = \frac{3 - {}{\Delta}u_1(x_0)}{2}(x \cdot \xi^{\perp})^2,
    \]
    where $\xi^{\perp}$ is a unit normal vector to the ray $\tilde{x}_0$. 
\end{corollary}
\begin{proof}
    Lemma \ref{lemma:reductiontotheobstacleproblem} and Proposition \ref{lemma:C^2obstacle} ensure that $v$ satisfy the hypotheses of Proposition \ref{prop:classificationofblow-ups}. 
    So, fix any blow-up $v_{\infty}$ of $v$ at $x_0$, and for convenience choose coordinates with $x_0 = 0$. Then, since $v = 0$ on $\tilde{x}_0$, $\Lambda(v_{\infty})$ must contain the line $\{t\xi \ | \ t < 0\}$. This implies that $\xi$ is a null eigenvector for the matrix $Q_{\infty}$ in Proposition \ref{prop:classificationofblow-ups}, and therefore $\xi^{\perp}$ is an eigenvector for $Q_{\infty}$ with eigenvalue $1$. Thus, $v_{\infty}$ has the desired form by Proposition \ref{prop:classificationofblow-ups}. 
\end{proof}

\subsection{Quantitative bounds on the density at singular points}

We now wish to get a quantitative bound on the rate at which the density of $\Lambda(v)$ goes to $0$ at a singular point. This is given by the following result from \cite{BlankHao15}. 

\begin{proposition}[Universal bound on the density at a singular point \cites{Blank01, BlankHao15}]\label{lemma:universalboundondensityatsingularpoint}
    Let $U \subset \mathbf{R}^n$ be a bounded domain and $\lambda>0$. 
    Suppose $v \in C^{1,1}(\overline{U})$ is a nonnegative function satisfying $\Delta v = f\characteristic{\{v > 0\}}$, where $f \in C(\overline{U})$ satisfies $f \geq \lambda > 0$. Then, the set of singular free boundary points $\mathcal{S} \subset F(v)$ is relatively closed, and there exists a  modulus of continuity $\varpi$ depending only on $\lambda$ and $n$
    with $\varpi(r) \to 0$ as $r \to 0^+$ such that for any singular point $x_0 \in F(v)$, we have
    \begin{equation}
      \label{eq:sing-point-density}
      \frac{|\Lambda(v) \cap B_{r}(x_0)|}{|B_{r}(x_0)|} < \varpi(r)
    \end{equation}
    for all $r > 0$.
\end{proposition}

\begin{proof}
    Since $f \in C(\overline{U})$, $f$ is bounded and uniformly continuous by Remark \ref{R:Heine-Cantor}, and therefore $f$ is of vanishing mean oscillation. So, we may apply the results of \cite{BlankHao15,Blank01}. 
    The fact that $\mathcal{S}$ is relatively closed is given by \cite{Blank01}*{Remark 4.9}, \cite{BlankHao15}*{Corollary 4.7}. For the uniform modulus of continuity, recall that the work of Blank and Hao \cite{BlankHao15}*{Remark 4.6} gives a  modulus of continuity $\varpi$ {depending only on $\lambda$ and $n$}
    that at $x_0 \in F(v)$, the existence of an $r > 0$ for which
    \[\frac{|\Lambda(v) \cap B_r(x_0)|}{|B_r(x_0)|} \geq \varpi(r) \]
    implies $x_0$ is a regular point. The estimate \eqref{eq:sing-point-density}  follows at singular points.
\end{proof}

\subsection{Singular points are discrete}

The main theorem of this section shows that set of singular points of the monopolist's free boundary is a discrete subset of the tame free boundary, i.e. there is no tame free boundary point which is the limit of singular tame free boundary points. Before proceeding, we need to show that every singular point is a local maximum for $R$ ---
and not just the accumulation point of such maxima as implied by 
\cite[Theorem 7.8]{McCannRankinZhang24+}.

\begin{lemma}[Singular points are local maxima]\label{lemma:singularpointsarelocalmaxima}
    Let $v$, $U$, and $u_1$ be defined as in Lemma \ref{lemma:reductiontotheobstacleproblem}, and suppose $x_0 = x(R(0), 0) \in F(v)$ is a singular point. Then, $0$ is a strict local maximum of $R$.
\end{lemma}
\begin{proof}
    For the sake of contradiction, suppose $x_{\infty} = x(R(0), 0) \in F(v)$ is a singular point that is not a strict local maximum of $R$, and let $\{t_k\}_{k=1}^{\infty} \subset (-\epsilon ,\epsilon)$ be a sequence of points with $t_k \to 0$ as $k \to \infty$ and $R(t_k) \geq R(0)$ for every $k$. Choose coordinates so that $x_{\infty} = 0$ and 
    \[
    \widetilde{x_{\infty}} = \{t e_2 \ |\ -R(0) \leq t \leq 0\},
    \]
    and write 
    \[
    x_k := x(R(t_k), t_k)
    \]
    in the $(r,t)$-coordinates of Remark \ref{remark:constructionof(r,t)coordinates}. Define
    \[
    y_k := x(R(0), t_k).
    \]
    Since $R(0) \leq R(t_k)$ for each $k$, we know that $y_k \in \widetilde{x_k}$ for each $k$, and consequently 
    \begin{equation}\label{eq:v(y_k)iszero}
        v(y_k) = 0 \qquad \text{for every $k$.}
    \end{equation}
 Now setting
    $r_k := |y_k|$,    
     we have for large $k$ that
    \[
    (\frac{y_k}{ r_k} \cdot e_1)^2 \geq c > 0
    \]
    for some $c$ depending only on the Lipschitz constant for the $(r,t)$-coordinate change.
    Moreover, $\frac{y_k}{r_k} \in \mathbf{S}^1$, and so possibly passing to a subsequence we may assume that $\frac{y_k}{r_k} \to y_{\infty}$ for some $y_{\infty} \in \mathbf{S}^1$. Clearly,
    \[(y_{\infty} \cdot e_1)^2 \geq c > 0.
    \]
    By Corollary \ref{cor:uniqueblow-upsatsingularpoints}, let \[
    v_{\infty}(x) := \frac{3 - \Delta u_1(0)}{2}(x \cdot e_1)^2
    \]
    be the unique blow-up of $v$ at $x_{\infty}$. Then, since $v_{r_k} \to v_{\infty}$ uniformly on compact sets, we have
    \begin{align*}
        0 &= \lim_{k \to \infty}\frac{1}{r_k^2}v(y_k) \\
        &= \lim_{k \to \infty}v_{r_k}(\frac{y_k}{r_k}) \\
        &= v_{\infty}(y_{\infty}) \\
        & = \frac{3 - \Delta u_1(0)}{2}(y_{\infty}\cdot e_1)^2 \\
        &> 0,
    \end{align*}
    a contradiction. 
\end{proof}

Now, we are ready to prove that the set of tame singular points is discrete;
for Example \ref{R:square} --- the square  of Rochet and Chon\'e \cites{RochetChone98, McCannRankinZhang24+} --- this means any singular points can only accumulate at two points of the fixed boundary and/or one or two points on $\Omega_0\cup \Omega_1^0$.

\begin{theorem}[Singular set is discrete]\label{theorem:singularsetisdiscrete}
    Let $u$ solve the monopolist's problem on a convex domain $X \subset \mathbf{R}^2$ satisfying hypotheses (a)--(c) of Theorem \ref{thm:mainresults}.     Let $\mathcal{T} \subset \Gamma$ be the set of tame free boundary points, 
    and let $\mathcal{S} \subset \mathcal{T}$ be the subset of singular tame free boundary points, 
    as in \eqref{Omega1+}--\eqref{eq:DefSingular}. Then, $\mathcal{S}$ is a discrete subset of $\mathcal{T}$ in the subspace topology.
\end{theorem}

\begin{proof}
     For the sake of contradiction, suppose $\{x_k\}_{k=1}^{\infty} \subset \mathcal{S}$ is a sequence of singular tame free boundary points converging to the tame free boundary point $x_{\infty} \in \mathcal{T}$. Since the set of singular points is relatively closed by Proposition \ref{lemma:universalboundondensityatsingularpoint}, $x_{\infty}$ is a singular point as well. Let $U$, $v$, and $u_1$ be defined as in Lemma \ref{lemma:reductiontotheobstacleproblem}, and without loss of generality assume $\{x_k\}_{k=1}^{\infty} \subset U$. For each $k$, write 
    \[
    x_k = x(R(t_k), t_k)
    \]
    in the $(r,t)$-coordinates of Remark \ref{remark:constructionof(r,t)coordinates}, and let $x_{\infty} = x(R(t_{\infty}), t_{\infty})$. By choosing coordinates and parametrization we may assume that $t_{\infty} = 0$, $x_{\infty} = 0$, and 
    \[
    \widetilde{x_{\infty}} = \{t e_2 \ |\ -R(0) \leq t \leq 0\}.
    \]
    By Lemma \ref{lemma:singularpointsarelocalmaxima}, $0$ is a local maximum of $R$, so by passing to a subsequence we may assume that $R(t_k) \leq R(0)$ for all $k$. Define
    \[
    y_k := x(R(t_k), 0),
    \]
    then since $R(t_k)\leq R(0)$ we have $y_k \in \widetilde{x_{\infty}}$ for all $k$. Set 
    \[
    r_k := |y_k - x_k|,
    \]
    and consider 
    \[
    v_{k}(x) := \frac{1}{r_k^2} v(x_k + r_k x).
    \]
    Since $v \in C^{1,1}(\overline{U})$, the Arzela-Ascoli theorem implies that there is some globally defined $v_{\infty} \in C^{1,1}(\mathbf{R}^n)$ and a subsequence (not relabelled) such that 
    \[
    v_{k} \to v_{\infty}
    \]
    and 
    \[
    \nabla v_{k} \to \nabla v_{\infty}
    \]
    uniformly on compact sets. Recall from \cite{Blank01}*{Theorem 2.10(a)} that $v_{\infty}$ is a global solution to the obstacle problem 
    \[
    \Delta v_{\infty} = (3 - \Delta u_1(0)) \characteristic{\{v_\infty > 0\}}, 
    \]
    so in particular $v_{\infty}$ is convex by \cite{Caffarelli98}*{Corollary 7}.

We claim
\begin{description}
\item[Claim 1] $v_\infty \equiv 0$ on $\{(0,s)\ | \ s < 0\}$.
\item[Claim 2] There is some $\gamma = (\gamma^1, \gamma^2) \in \mathbf{S}^1$ with  $\gamma^1 \ne 0$ such that $v_{\infty}(\gamma) = 0$. 
  \item[Claim 3] $|\Lambda(v_{\infty})| = 0$. 
  \end{description}
  Of course, Claims 1 and 2 combined with the fact that $\Lambda(v_{\infty})$ is convex contradict Claim 3. It remains to establish Claims 1--3. 

    Recall from Remark \ref{remark:constructionof(r,t)coordinates} that $\xi(t_k)$ is the unit vector parallel to the ray $\widetilde{x_k}$ such that $x_k + \delta \xi(t_k) \in \Omega(v)$ for all sufficiently small $\delta > 0$. The fact that $\xi$ is Lipschitz \cite{McCannRankinZhang24+}*{Lemma 6.5} implies that $\xi(t_k) \to \xi(0) = e_2$ as $k \to \infty$.
    
  \textbf{Claim 1. } Let $s < 0$. By the uniform convergence of $v_k \rightarrow v_\infty$ we have
  \begin{align*}
    v_\infty(0,s) &= \lim_{k\rightarrow \infty }v_k(0,s) \\
               &= \lim_{k\rightarrow \infty }v_k(s\xi(t_k)) \quad \quad \text{ since $s\xi(t_k) \rightarrow (0, s)$}\\
                 &= 0, 
  \end{align*}
  where the final equality is because $x_k + r_ks\xi(t_k) \in \widetilde{x_k}$ for $s < 0$. 

  \textbf{Claim 2.}  For each $k$, define 
  \[
  \gamma_k := \frac{y_k - x_k}{r_k} \in \mathbf{S}^1,
  \]
  then since the coordinate change to $(r,t)$-coordinates is bi-Lipschitz, there is some $c > 0$ depending only on this Lipschitz constant so that for large enough $k$,
  \[
  (\gamma_k\cdot e_1)^2 \geq c > 0. 
  \]
  Possibly passing to a subsequence, there is some $\gamma_{\infty} \in \mathbf{S}^1$ such that 
  \[
  \gamma_k \to \gamma_{\infty}.
  \]
  Then, we have 
  \[
  (\gamma_{\infty} \cdot e_1)^2 \geq c > 0,
  \]
  so if we write $\gamma_{\infty} = (\gamma^1, \gamma^2)$, we have $\gamma^1 \ne 0$. Moreover, since $y_k \in \widetilde{x_{\infty}}$ for each $k$, we have
  \begin{align*}
      v_{\infty}(\gamma_{\infty}) &= \lim_{k \to \infty}v_{r_k}(\gamma_k) \\
      &= \lim_{k \to \infty}\frac{1}{r_k^2}v(x_k + r_k \gamma_k) \\
      &= \lim_{k \to \infty}\frac{1}{r_k^2}v(y_k)\\
      &= 0,
  \end{align*}
  as desired. 
  
  \textbf{Claim 3. } 
    Finally, we wish to show that $|\Lambda(v_{\infty})| = 0$. By Proposition \ref{lemma:universalboundondensityatsingularpoint}, for any $R > 0$ we have
    \begin{align*}
        \frac{|\Lambda(v_k) \cap B_R|}{|B_R|} &= \frac{|\Lambda(v) \cap B_{Rr_k}(x_k)|}{|B_{Rr_k}(x_k)|} \\
        &\leq \varpi(Rr_k).
    \end{align*}
    Applying Lemma \ref{lemma:measurestabilityforblowups}, we see that 
    \begin{align*}
        \lim_{k \to \infty}\big| |\Lambda(v_k) \cap B_R| - |\Lambda(v_{\infty}) \cap B_R|\big| &\leq \lim_{k \to \infty}|\big(\Lambda(v_k) \Delta \Lambda(v_{\infty}) \big)\cap B_R| \\
        &= 0,
    \end{align*}
    and therefore 
    \begin{align*}
        \frac{|\Lambda(v_{\infty}) \cap B_R|}{|B_R|} &= \lim_{k \to \infty}\frac{|\Lambda(v_k) \cap B_R|}{|B_R|} \\
        &\leq \lim_{k \to \infty}\varpi(Rr_k) \\
        &= 0.
    \end{align*}
    Since $R$ was arbitrary, we conclude that $|\Lambda(v_{\infty})| = 0$. 
\end{proof}

\section{Regular points on rough free boundaries}

In this section, we study the regular points of the free boundary for solutions to the obstacle problem with $C^2$ obstacle. Thanks to Blank's example \cite{Blank01}*{Theorem 7.3}, we know that in this scenario, %we cannot even expect 
the free boundary need not be locally the graph of a function. Nonetheless, using the results of \cites{Blank01, BlankHao15}, we prove that the regular free boundary is a $C^{\alpha}$-submanifold for any $\alpha \in (0,1)$, which combined with Theorem \ref{theorem:singularsetisdiscrete} yields the sharp Hausdorff dimension bound for the free boundary.

\subsection{H\"{o}lder regularity at regular points}
As a simple application of the work of Blank \cite{Blank01}, we get H\"{o}lder regularity of the free boundary at regular points. Our main input will be the classical Reifenberg theorem \cite{Reifenberg60}. 

Denote by $d_{\mathcal{H}}$ the Hausdorff distance, so if $A, B \subset \mathbf{R}^n$, 
\[
d_{\mathcal{H}}(A,B) := \max\{\sup_{a \in A}\dist(a, B), \sup_{b \in B} \dist(b,A)\}.
\]
Denote by $\mathbf{Gr}(n-1, n)$ the space of hyperplanes through the origin in $\mathbf{R}^n$. 
Given a set $V \subset \mathbf{R}^n$ and a compact set $K \subset \mathbf{R}^n$, we will define the Reifenberg modulus of flatness by 
\begin{equation}\label{eq:moduleofflatnessdefinition}
    \vartheta_{V}(r; K) := r^{-1}\sup_{x \in V \cap K}\left(\inf_{\Gamma \in \mathbf{Gr}(n-1, n)}d_{\mathcal{H}}\big(V \cap B_{r}(x), (\Gamma + x) \cap B_{r}(x)\big)\right);
\end{equation}
c.f. \cite{Blank01}*{(6.3)}, \cite{BlankHao15}*{(5.3)}, \cite{Gigli25}. The set $V$ will be called $\delta$-Reifenberg flat provided there is $r_\delta$ with $\vartheta_V(r;K) \leq 2\delta$ for all $r < r_\delta$. The set $V$ will be called \textit{Reifenberg vanishing} if $\vartheta_{V}(r; K) \to 0$ as $r \to 0$ for any compact $K$. 

The main input we will need from \cites{Blank01, BlankHao15} is the following result, which is proven in \cite{Blank01}*{Theorem 7.1}, \cite{BlankHao15}*{Theorem 5.6}:
\begin{theorem}[Reifenberg vanishing at regular points \cites{Blank01, BlankHao15}]\label{theorem:Reifenbergvanishingatregularpoints}
    Let $U \subset \mathbf{R}^n$ be a bounded domain, and suppose $v \in C^{1,1}(\overline{U})$ solves the obstacle problem $\Delta v = f \characteristic{\{v > 0\}}$, where $f \in C(\overline{U})$ satisfies $f \geq c_0 > 0$. Let $\mathcal{R} \subset F(v)$ denote the set of regular free boundary points, and suppose $K \Subset \mathcal{R}$. Then, 
    \[
    \vartheta_{\mathcal{R}}(r; K) \to 0  \text{ as $r \to 0^+$. }
    \]
\end{theorem}

By combining this fact with the classical Reifenberg theorem, we get H\"{o}lder 
continuity of the free boundary in a neighbourhood of any regular point,
for the solution of any 
obstacle problem with a $C^2$ obstacle. 

   \begin{lemma}[Free boundary is a H\"{o}lder submanifold at regular points]\label{lemma:freeboundaryisholdersubmanifold}
    Let $U \subset \mathbf{R}^n$ be a domain, and suppose $v \in C^{1,1}(\overline{U})$ solves the obstacle problem $\Delta v =  f\characteristic{\{v > 0\}}$, where $f \in C(\overline{U})$ satisfies $f \geq c_0 > 0$. Then, for each regular point $x \in F(v)$, there is a neighbourhood $x \in V \subset F(v)$ such that for every $\alpha \in (0,1)$ there exists a bijection 
    \[
    \Phi: B^{n-1} \to V
    \]
    from the $(n-1)$-ball $B^{n-1}\subset \mathbf{R}^{n-1}$ of radius $1$,
    such that $\Phi$ and $\Phi^{-1}$ are $\alpha$-H\"{o}lder continuous. 
  \end{lemma}

\begin{proof}
  Let $x \in F(v)$ be a regular point. By \cite{Blank01}*{Remark 4.9}, the set of regular points of $F(v)$ are relatively open.
  So, by Theorem \ref{theorem:Reifenbergvanishingatregularpoints}, we conclude that if $V \Subset F(v)$ is a sufficiently small neighbourhood of $x$ so that $\overline{V}$ is contained in the set of regular points, then $V$ is Reifenberg vanishing and, in particular, $\delta$-Reifenberg flat for every $\delta >0$.  The classical Reifenberg theorem \cite{Blank01}*{Theorem 6.2}, \cite{Gigli25}*{Theorem 1.1}, \cite{Reifenberg60}, implies that because $V$ is $\delta$-Reifenberg flat  there is an $\alpha = \alpha(\delta)$ and an $\alpha$-bi-H\"{o}lder homeomorphism $\Phi$ between an open subset of $\mathbf{R}^{n-1}$ and $V$. Moreover, $\alpha(\delta) \to 1$ as $\delta \to 0$, so we may take  $\alpha \in (0,1)$ arbitrarily close to $1$. By possibly taking $V$ to be slightly smaller and choosing coordinates for $\mathbf{R}^n$, we get the desired map 
    \[
    \Phi:B^{n-1} \to V. 
    \]
\end{proof}

\begin{corollary}[Hausdorff dimension of regular points]\label{cor:Hausdorffdimensionofregularpoints}
    Let $U \subset \mathbf{R}^n$ be a domain, and suppose $v \in C^{1,1}(\overline{U})$ solves the obstacle problem $\Delta v =  f\characteristic{\{v > 0\}}$, where $f \in C(\overline{U})$ satisfies $f \geq c_0 > 0$. Let $\mathcal{R} \subset F(v)$ denote the set of regular free boundary points. Then, 
    \[
    \mathrm{dim}_{\mathcal{H}}(\mathcal{R}) \leq n-1. 
    \]
\end{corollary}
\begin{proof}
 Fix $\alpha < 1$ and take $V$ and $\Phi$ as in Lemma \ref{lemma:freeboundaryisholdersubmanifold}, then since $\Phi$ is $\alpha$-H\"{o}lder we have by \cite{Falconer03}*{Proposition 2.3} that
    \[
    \dim_{\mathcal{H}}(V) \leq \frac{1}{\alpha}\dim_{\mathcal{H}}(B^{n-1}).
    \]
    Since we may take $\alpha$ arbitrarily close to $1$, we conclude that
    \[
    \mathrm{dim}_{\mathcal{H}}(V) \leq n-1.
    \]
    Since $\mathcal{R}$ can be covered by countably many such neighbourhoods, we conclude that 
    \[
    \mathrm{dim}_{\mathcal{H}}(\mathcal{R}) \leq n-1. 
    \]
\end{proof}

Combining this with Theorem \ref{theorem:singularsetisdiscrete}, we get
\begin{corollary}[H\"{o}lder regularity and sharp dimension bound]\label{cor:holderregularityandsharpdimensionbound}
    Let $u$ solve the monopolist's problem on a convex domain $X \subset \mathbf{R}^2$ satisfying hypotheses (a)--(c) of Theorem \ref{thm:mainresults}. 
    Let $\mathcal{T} \subset \Gamma$ be the set of tame free boundary points, and $\mathcal{S}$ the subset of singular points,
    as in \eqref{Omega1+}--\eqref{eq:DefSingular}. Then, any connected component of  $\mathcal{T} \setminus \mathcal{S}$ is an $\alpha$-H\"{o}lder submanifold for any $\alpha \in (0,1)$, and 
    \[
    \mathrm{dim}_{\mathcal{H}}(\mathcal{T}) \leq 1.
    \]
\end{corollary}
\begin{proof}
    This follows immediately from Lemma \ref{lemma:freeboundaryisholdersubmanifold} and Corollary \ref{cor:Hausdorffdimensionofregularpoints} since the set of singular points is discrete by Theorem \ref{theorem:singularsetisdiscrete}. 
\end{proof}

\begin{remark}[Parameterization caveat]
    \label{R:not Dini}
    Lemma \ref{lemma:freeboundaryisholdersubmanifold}
    need not imply H\"older (nor Dini) continuity of $R(\cdot)$.
    Otherwise  $\Delta u$ would be H\"older (or Dini) continuous by Proposition~\ref{lemma:C^2obstacle}, so the additional hypothesis of Theorem~\ref{thm:mainresults}\eqref{item:Dini partial regularity} would be satisfied. This is enough to conclude $C^1_{\mathrm{loc}}$-smoothness of $\mathcal T \setminus \mathcal S$, as we will next demonstrate.
\end{remark}

\subsection{Dini condition gives partial regularity at tame free boundary points}

Without additional control on the rate of vanishing of the Reifenberg modulus of flatness, we do not know whether the hyperplanes $\Gamma$ attaining the infimum in \eqref{eq:moduleofflatnessdefinition} converge to some limit as $r \to 0$; this is an obstruction to trying to obtain an improvement upon the H\"{o}lder regularity guaranteed by the Reifenberg theorem. However, when $\Delta u_1$ is Dini continuous, Blank~\cite{Blank01} shows that the modulus of flatness \eqref{eq:moduleofflatnessdefinition} also satisfies the Dini condition, allowing him to prove $C^1$ regularity at regular points \cite{Blank01}*{Theorem 0.1}. We will combine this result with the bootstrapping in \cite{McCannRankinZhang24+} to prove a conditional partial regularity result for the tame free boundary. 
By Proposition \ref{lemma:C^2obstacle},
notice that $\Delta u_1$ is Dini continuous if $R$ is Dini continuous. 
Recalling the definitions of $\mathcal S$ and $\mathcal T$ from \eqref{Omega1+}--\eqref{eq:DefSingular}:

\begin{proposition}[Partial regularity of the tame free boundary]\label{prop:conditionalregularity}
Take $u_1$ to be defined as in Lemma \ref{lemma:reductiontotheobstacleproblem}. 
\begin{enumerate}
\item\label{I:C1} If a relatively open subset $A \subseteq \mathcal{T}$ is a $C^1$ curve, then $A$ is $C^{\infty}$ outside of a closed, nowhere dense subset of $\mathcal T$. 
\item\label{I:Dini} If instead $\Delta u_1$ is Dini continuous, then $\mathcal{T}$ is a $C^1$ curve except at the relatively discrete set $\mathcal{S}$, and is $C^{\infty}$ outside of a closed, nowhere dense subset of $\mathcal{T}$.
\end{enumerate}
\end{proposition}
\begin{proof}

    \eqref{I:C1} Suppose some relatively open $A \subseteq \mathcal{T}$ is a $C^1$ curve. Choose the tangent vector $\tau : A \to \mathbf{S}^1$ continuously. Since the ray direction $\xi$ is Lipschitz by \cite{McCannRankinZhang24+}*{Lemma 6.5}, the map $x \mapsto |\xi(x) \cdot \tau(x)|$ is continuous, so the set
    \[
    \mathcal{C} := \{x \in A  \ | \ |\tau(x)\cdot \xi(x)| \ne 1 \}
    \]
    is open. Moreover, since $F(v)$ is the graph of $R$ in the bi-Lipschitz $(r,t)$-coordinates of Remark \ref{remark:constructionof(r,t)coordinates}, we see that $R$ is Lipschitz in a neighbourhood of every $x_0 \in \mathcal{C}$, and thus every connected component of $\mathcal{C}$ is a $C^{\infty}$ curve by the bootstrapping in \cite{McCannRankinZhang24+}. 

    It remains to consider the (relatively) closed set $A \setminus \mathcal{C}$. We claim that this set has empty interior. Indeed, suppose that $W \subset A \setminus \mathcal{C}$ is open and connected. Since $A$ is a $C^1$ curve, we may find a $C^1$ unit-speed parametrization $\eta: (-\delta, \delta) \to W$ of $W$. By definition of $\tau$, we can assume $\eta$ is parameterized so that
    \[
    \eta'(t) = \tau(\eta(t))
    \]
    for each $t \in (-\delta, \delta)$. But since $W \subset A \setminus \mathcal{C}$, this implies that
    \[
    |\eta'(t) \cdot \xi(\eta(t))| = 1
    \]
    for every $t \in (-\delta, \delta)$. So, since $\eta$ is $C^1$ and $\xi$ is Lipschitz, we can reparameterize $\eta$ to ensure that for every $t \in (-\delta, \delta)$,
    \[
    \eta'(t) = \xi(\eta(t)). 
    \]
    By definition of $\xi$, we see that the function
    \[
    f(t) = \eta(0) + t \xi(\eta(0))
    \]
    also satisfies the ordinary differential equation
    \[
    \bigg\{ \begin{matrix}
        f(0) = \eta(0), \\
        f'(t) = \xi(f(t)).
    \end{matrix}
    \]
    So, since $\xi$ is Lipschitz, we may apply the Picard-Lindel\"{o}f theorem to conclude that $\eta(t) = f(t)$. Therefore, for any $x_0 \in W$, we have that $W$ is contained entirely in the ray $\widetilde{x_0}$, contradicting the continuity of $R$. 

\eqref{I:Dini} Now, assume $\Delta u_1$ is Dini continuous. Since Dini continuity of $\Delta u_1$ implies that $F(v)$ is $C^{1}$ in a neighbourhood of any regular point \cite{Blank01}*{Theorem 0.1}, it follows that $\mathcal{T}$ is a $C^1$ curve except at the relatively discrete subset $\mathcal S \subset \mathcal T$ of singular points. We have therefore shown that every connected component $A$ of $\mathcal{T} \setminus \mathcal{S}$ is $C^{\infty}$ outside of a closed, nowhere dense subset of $\mathcal{T}$, so we conclude that $\mathcal{T}$ is $C^{\infty}$ except at a closed, nowhere dense subset.
\end{proof}

The result of Proposition \ref{prop:conditionalregularity}\eqref{I:C1}
becomes unconditional in view of Remark \ref{R:revision}.

\bibliographystyle{plain}\bibliography{bibliography.bib}

\end{document}